\documentclass[11pt,reqno]{amsart}
\usepackage{graphicx,amsmath,amsfonts,latexsym,amssymb,amsthm,color}
\usepackage{latexsym}
\usepackage{verbatim}
\usepackage{enumerate}
\usepackage{enumitem}

\usepackage[a4paper]{geometry}
\definecolor{blau}{rgb}{0.05,0.2,0.7}
\definecolor{auchblau}{rgb}{0.03,0.3,0.7}
\usepackage[pdftex,linkbordercolor={0.8 0.5 0.2}]{hyperref} 

\renewcommand{\Re}{\operatorname{Re}}
\renewcommand{\Im}{\operatorname{Im}}

\newcommand{\Lsp}{\operatorname{Lsp}}

\newcommand{\Sing}{\operatorname{sing\, supp}}

\DeclareMathOperator{\supp}{supp}
\newcommand{\der}{\mathrm{d}}
\newcommand{\rmi}{\mathrm{i}}
\newcommand{\ee}{\mathrm{e} }
\newcommand{\sphere}{{\mathbb{S}^{d-1}}}
\newcommand{\calO}{\mathcal{O}}

\newcommand{\tr}{\mathrm{Tr}}
\newcommand{\Tr}{\mathrm{Tr}}

\newcommand{\rel}{\mathrm{rel}}	
\newcommand{\dist}{\mathrm{dist}}

\newcommand{\R}{\mathbb{R}}
\newcommand{\C}{\mathbb{C}}

\newcommand{\cf}{\mathcal{N}}

\newcommand{\loc}{\mathrm{loc}}

\newcommand{\dd}{\mathrm{d}}

\setlist{leftmargin=8mm}

\newtheorem{theorem}{Theorem}[section]
\newtheorem{definition}[theorem]{Definition}

\newtheorem{corollary}[theorem]{Corollary}
\newtheorem{proposition}[theorem]{Proposition}

\newtheorem{rem}[theorem]{Remark}

\def\mathbi#1{\textbf{\em #1}}

\title[Trace Singularities and the Poisson relation]{Trace singularities in obstacle scattering and the Poisson relation for the relative trace}

\author[Y. Fang]{Yan-Long Fang}
\address{School of Mathematics,  University of Leeds,  Leeds , Yorkshire, LS2 9JT,
UK} \email{y.l.fang@leeds.ac.uk}

\author[A. Strohmaier]{Alexander Strohmaier}
\address{School of Mathematics,  University of Leeds,  Leeds , Yorkshire, LS2 9JT,
UK} \email{a.strohmaier@leeds.ac.uk}

\thanks{Supported by Leverhulme grant RPG-2017-329}

\begin{document}

\begin{abstract} 
 We consider the case of scattering by several obstacles in $\R^d$, $d \geq 2$ for the Laplace operator $\Delta$ with Dirichlet boundary conditions imposed on the obstacles. In the case of two obstacles, we have the Laplace operators $\Delta_1$ and $\Delta_2$ obtained by imposing Dirichlet boundary conditions only on one of the objects. The relative operator
 $g(\Delta) - g(\Delta_1) - g(\Delta_2) + g(\Delta_0)$ was introduced in \cite{RTF} and shown to be trace-class for a large class of functions $g$, including certain functions of polynomial growth. When $g$ is sufficiently regular at zero and fast decaying at infinity then, by the Birman-Krein formula, this trace can be computed from the relative spectral shift function $\xi_\rel(\lambda) = -\frac{1}{\pi} \Im(\Xi(\lambda))$, where  $\Xi(\lambda)$
 is holomorphic in the upper half-plane and fast decaying.
 In this paper we study the wave-trace contributions to the singularities of the Fourier transform of $\xi_\rel$. In particular we prove that
 $\hat \xi_\rel$ is real-analytic near zero and we relate the decay of $\Xi(\lambda)$ along the imaginary axis to the first wave-trace invariant of the shortest bouncing ball orbit between the obstacles. The function $\Xi(\lambda)$ is important in the physics of quantum fields as it determines the Casimir interactions between the objects.

\medskip

\noindent\textsc{R\'esum\'e.}

Nous considérons pour le laplacien $\Delta$ la diffusion par plusieurs obstacles dans $\mathbb R^d$, $d \ge 2$, munis de la condition aux limites de Dirichlet. Lorsqu'il y a deux obstacles, nous dénotons $\Delta_1$ et $\Delta_2$ les laplaciens obtenus en imposant la condition aux limites de Dirichlet sur un seul des objets. L'opérateur de trace relative $g(\Delta) - g(\Delta_1) - g(\Delta_2) + g(\Delta_0)$ a été introduit dans \cite{RTF} et s'avère être un opérateur à trace pour une grande classe de fonctions $g$, dont certaines fonctions à croissance polynomiale. Lorsque $g$ est suffisamment régulier en zéro et décroît rapidement à l'infini, la formule de Birman--Krein permet de calculer cette trace à partir de la fonction de décalage spectral $\xi_{rel}(\lambda) = - \frac 1 \pi \Im(\Xi(\lambda))$, où $\Xi$ est une fonction holomorphe à décroissance rapide dans le demi-plan supérieur. Dans cet article, nous étudions les contributions de la trace des ondes aux singularités de la transformée de Fourier de $\xi_{rel}$. Nous démontrons entre autres que $\hat \xi_{rel}$ est une fonction analytique réelle près de zéro, et nous relions la décroissance de $\Xi(\lambda)$ le long de l'axe imaginaire au premier invariant de la trace des ondes correspondant aux trajectoires rebondissant entre les deux obstacles. La fonction $\Xi(\lambda)$ est importante en théorie quantique des champs car elle détermine les interactions de Casimir entre les objets.

\end{abstract}

\maketitle


\section{Introduction}
\label{Introduction}

We consider obstacle scattering for the Laplace operator $\Delta = \der^* \der= \nabla^* \nabla$ acting on functions on $d$-dimensional Euclidean space $\R^d$ with $d \geq 2$.

\begin{figure}[h]
	\centering
	\includegraphics[clip, width=4cm]{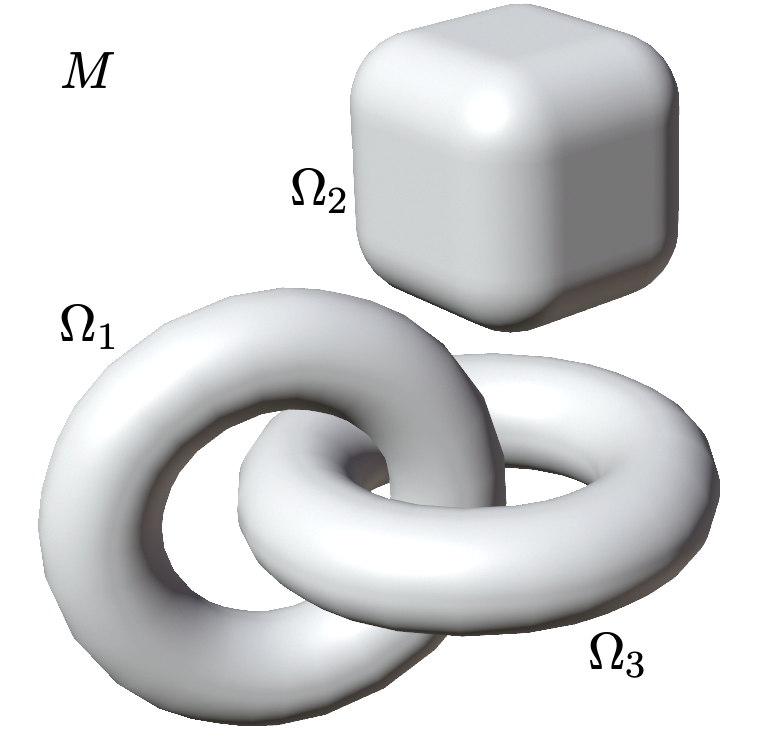}
	\caption{Three obstacles $\Omega_1,\Omega_2,\Omega_3$ in $\R^3$ with complement $M$.}
	\label{threeobstacles}
\end{figure}

Let $\Omega$ be a bounded open subset of $\R^d$ with smooth boundary such that $M = \R^d \setminus \overline{\Omega}$ is connected.
The domain $\Omega$ will be assumed to consist of $N$ many connected components $\Omega_1,\ldots,\Omega_N$. 
The space $X = \R^d \setminus \partial \Omega$ therefore consists of the $N+1$-many connected components
$\Omega_1,\ldots,\Omega_N,M$.  We think of $\Omega$ as obstacles placed in $\R^d$. The open subset $M$ then corresponds to the exterior region of these obstacles. Similarly, we define $M_i = \R^d \setminus \overline{\Omega_i}$ and $X_i=\R^d \setminus \partial \Omega_i$.

\begin{definition}\label{Laplace operators}

The self-adjoint operators $\Delta_0, \Delta_M, \Delta_X, \Delta_{X_i}$ are defined as follows.
\begin{enumerate}
	 \item On the Hilbert space $L^2(\R^d)$ the free Laplacian $\Delta_0$ is defined as the self-adjoint operator $\Delta$ with domain $H^2(\R^d)$.
 \item On the Hilbert space $L^2(M)$ the exterior Laplacian $\Delta_M$ is the self-adjoint operator $\Delta$ with domain $H^2(M) \cap H^1_0(M)$, i.e. the operator constructed from the Dirichlet quadratic form on $C_0^\infty(M)$.
 \item On the Hilbert space $L^2(\Omega)$ the interior Laplacians $\Delta_{\Omega}$ with domain $H^2(\Omega) \cap H^1_0(\Omega)$ is constructed from the Dirichlet quadratic forms on $H^1_0(\Omega)$. This operator splits into a direct sum
 $\Delta_{\Omega} = \Delta_{\Omega_1} \oplus \ldots \oplus  \Delta_{\Omega_N}$ on 
 $L^2(\Omega)= L^2(\Omega_1) \oplus \ldots \oplus L^2(\Omega_N)$.
\item  On the Hilbert space $L^2(\R^d) = L^2(M) \oplus L^2(\Omega)$ is defined as the operator $\Delta_X = \Delta_M \oplus \Delta_\Omega$.
\item  On the Hilbert space $L^2(\R^d)$ the operator $\Delta_{X_i}$ the operator is defined by the Dirichlet quadratic form
on $H^1_0(X_i)$. This operator is the direct sum of $\Delta_{\Omega_i}$ and the exterior Dirichlet Laplacian $\Delta_{M_i}$.
\end{enumerate}
\end{definition}

Spectral and scattering theory describe the spectral resolution of these operators, which we now explain in more detail. A similar description as below is true in the more general black-box formalism in scattering theory as introduced by Sj\"ostrand and Zworski \cite{SZ1991} and follows from the meromorphic continuation of the resolvent and its consequences. The description below follows \cite{OS} and we refer the reader to this article for the details of the spectral decomposition and properties of the scattering matrix.

The operators $\Delta_{\Omega}$ and $\Delta_{\Omega_i}$ have purely discrete spectrum, whereas
$\Delta_M$ has purely absolutely continuous spectrum. The spectral resolution of $\Delta_\Omega$ is described by an orthonormal basis $(\phi_j)$ of smooth eigenfunctions $\phi_j \in C^\infty(\overline{\Omega})$ with eigenvalues
$\lambda_j^2$, where we assume
$$
 0<\lambda_1 \leq \ldots \leq \lambda_n \leq \lambda_{n+1} \leq \ldots
$$
The eigenvalue counting function $\cf_{\Omega}$ is defined by $\cf_\Omega(\lambda) = \#\{ \lambda_j < \lambda\}$ and satisfies a Weyl-law $$\cf_\Omega(\lambda) \sim C_d |\Omega| \lambda^{d}$$ for $\lambda \to +\infty$, where $C_d=(2\pi)^{-d}\omega_d$ and $\omega_d$ is the Euclidean volume of the unit ball in $\R^d$. 
The continuous spectral resolution of $\Delta_M$ is described by generalised eigenfunctions $E_\lambda(\Phi) \in C^\infty(\overline M)$ indexed by $\lambda>0$ and $\Phi \in C^\infty(\sphere)$. The are uniquely determined by the following properties 
\begin{enumerate}
\item $(\Delta - \lambda^2 ) E_\lambda(\Phi) =0$,
\item $E_\lambda(\Phi)|_{\partial M}=0$
\item The asymptotic expansion
$$
 E_\lambda(\Phi) = \frac{\ee^{-\rmi \lambda r} \ee^{\frac{i\pi(d-1)}{4}} }{r^{\frac{d-1}{2}}} \Phi + \frac{\ee^{\rmi \lambda r} \ee^{-\frac{i\pi(d-1)}{4}} }{r^{\frac{d-1}{2}}} \Psi_\lambda + O\left(\frac{1}{r^{\frac{d+1}{2}}}\right),  \quad \textrm{for} \,\,\,r \to \infty
$$
holds for some $\Psi_\lambda \in C^\infty(\sphere)$.
\end{enumerate}
As a result $\Psi_\lambda$ is uniquely determined and implicitly defines a linear mapping 
\begin{gather*}
 \mathbi{S}_\lambda:  C^\infty(\sphere) \to C^\infty(\sphere), \quad \Phi \mapsto \tau \Psi_\lambda,
\end{gather*}
where $\tau: C^\infty(\sphere) \to C^\infty(\sphere)$ is the pull-back of the antipodal map. The map $\mathbi{S}_\lambda : C^\infty(\sphere) \to C^\infty(\sphere)$ is called the scattering matrix,
and $ \mathbi{A}_\lambda= \mathbi{S}_\lambda - \mathrm{id}$ is called the scattering amplitude. The scattering matrix extends to a unitary operator 
$\mathbi{S}_\lambda : L^2(\sphere) \to L^2(\sphere)$ for $\lambda >0$, and 
has the following properties depending on whether the dimension is even or odd.
\begin{itemize}
 \item In case $d$ is odd the scattering matrix $\mathbi{S}_\lambda$ extends to a meromorphic function on $\C$ which is regular on the real line.
  It satisfies the functional relation $\mathbi{S}_\lambda \tau \mathbi{S}_{-\lambda} = \tau$ and is unitary in the sense that $(\mathbi{S}_\lambda)^{-1} = \mathbi{S}_{\overline{\lambda}}^*$.
 \item In case $d$ is even the scattering matrix $\mathbi{S}_\lambda$ extends to a meromorphic function on the logarithmic cover of $\C \setminus \{0\}$. It is holomorphic in the upper half-plane and regular on $\R \setminus \{0\}$. We have a functional relation 
 $\mathbi{S}_\lambda \, \tau \, (2 \, \mathrm{id} - \mathbi{S}_{-\lambda}) = \tau$, where $-\lambda$ is interpreted as $e^{\rmi \pi} \lambda$. Unitarity holds in the sense that
 $(\mathbi{S}_\lambda)^{-1} = \mathbi{S}_{\overline{\lambda}}^*$.
\end{itemize}

It can be shown that $\mathbi{A}_\lambda$ extends to a continuous family of trace-class operators on the real line and one has the following estimate on 
the trace norm 
\begin{equation*}
\| \mathbi{A}_\lambda \|_1 = \left \{  \begin{matrix} O(\lambda^{d-2}) & \textrm{ for } d\geq 3, \\ O(\frac{1}{-\log(\lambda)}) & \textrm{ for } d=2 \end{matrix} \right.
\end{equation*}
for all $| \lambda |<\frac{1}{2}$ in a fixed sector in the logarithmic cover of the complex plane, c.f. \cite[Theorem 1.11]{OS} or  \cite[Lemma 2.5]{christiansen1999weyl} in case $d \geq 3$.

If $f \in \mathcal{S}(\R)$ is a Schwartz function with $f(\lambda) = f(-\lambda)$ we have that 
$f(\Delta_X^{\frac{1}{2}}) -  f(\Delta_0^{\frac{1}{2}})$ is a trace-class operator with trace equal to
$$
 \Tr\left( f(\Delta_X^{\frac{1}{2}}) -  f(\Delta_0^{\frac{1}{2}})\right) = -\int_0^\infty \xi(\lambda) f'(\lambda) \der \lambda.
$$
for a function $\xi \in L^1_{\loc}(\R)$ with $\xi(\lambda)=0$ if $\lambda<0$. The function is uniquely determined and called the spectral shift function. The Birman-Krein formula \cite{BirmanKrein} states that for $\lambda>0$ one has
\begin{equation*}
\label{ssfunction}
\xi(\lambda) =  \xi_{ac}(\lambda) + \xi_p(\lambda)  = \frac{1}{2\pi \rmi} \log \det(\mathbi{S}_\lambda) + N_\Omega(\lambda).
\end{equation*}
The relative trace and the relative trace formula were introduced in \cite{RTF}. 
Suppose that $h$ is a polynomially bounded function and $\Delta_{X_j}$ as defined in Definition \ref{Laplace operators}. Then each of the operators $h(\Delta_X)$, $h(\Delta_{X_j})$,
and $h(\Delta_0)$ has $C^\infty_0(X)$ contained in its domain. We define
\begin{equation}
\label{RelR}
\mathcal{R}_h = h(\Delta_X^{\frac{1}{2}}) - h(\Delta_0^{\frac{1}{2}}) - \sum_{j=1}^N \left( h(\Delta_{X_j}^{\frac{1}{2}}) - h(\Delta_0^{\frac{1}{2}}) \right).
\end{equation}
This operator has dense domain, containing $C^\infty_0(X)$. Whenever this operator is bounded (which is the only case we consider in this paper), we will denote its closure also by $\mathcal{R}_h$.
If $h \in \mathcal{S}(\R)$ is even then $\mathcal{R}_h$ is a trace-class operator and
$$
 \Tr (\mathcal{R}_h) = - \int_0^\infty \xi_\rel(\lambda) h'(\lambda) \der \lambda,
$$
where the relative spectral shift-function $\xi_\rel$ is for $\lambda>0$ given by
$$
 \xi_\rel(\lambda)= \xi(\lambda) - \sum_{j=1}^N \xi_j(\lambda) =\frac{1}{2\pi \rmi} \log \frac{\det(\mathbi{S}_\lambda)}{\det(\mathbi{S}_{1,\lambda}) \cdots \det(\mathbi{S}_{N,\lambda})}.
$$
Here $\det(\mathbi{S}_{j,\lambda})$ correspond to the scattering matrices when only obstacle $\Omega_j$ is present, and the other obstacles are removed. Note that the interior contributions cancel out.
The main result of \cite{RTF} is that $\mathcal{R}_h$ is trace-class for a much larger class of functions including $h(\lambda)=\lambda^s$ for $\Re(s) >0$. We briefly recall the result for a slightly less general class of functions that are sufficient for the purposes of this article.
Define
$$
\mathfrak{S}_\theta=\{z \in \C \; | \; z \ne 0, |\arg (z)|< \theta \}
$$
for some $0<\theta\le \pi$ and let $\mathcal{P}_\theta$ be the set of functions that are polynomially bounded, holomorphic in $\mathfrak{S}_\theta$, and satisfy
the estimate $| g(z)| = O(|z|^a)$ if $|z|<1$ for some $a>0$.
If $f$ is a function such that $f(\lambda) = g(\lambda^2)$ for some $g$ in $\mathcal{P}_\theta$ then $\mathcal{R}_f$ is trace-class and its trace can be computed by
$$
 \Tr (\mathcal{R}_f) = \frac{\rmi}{2 \pi} \int_\Gamma f'(\lambda) \Xi(\lambda) \der \lambda,
$$ 
for some universal function $\Xi$ that is independent of $f$.
Here $\Gamma$ is the path in the complex plane consisting of the rays $(-\infty,0] \to \C, t \mapsto -t e^{\rmi \theta/2}$
and $[0,\infty) \to \C, t \mapsto t e^{\rmi (\pi-\frac{\theta}{2})}$.
The function $\Xi$ is holomorphic in the upper half plane and satisfies on
$\mathfrak{D}_\epsilon=\{ \lambda \in \C \mid \Im(\lambda) > \epsilon |\lambda| \}$ the following bounds
 \begin{align}
 \label{roughbounds}
  |  \Xi(\lambda)| &\leq C_{\delta',\epsilon}e^{- \delta' \Im{\lambda}}, &
  |  \Xi'(\lambda) | &\leq C_{\delta',\epsilon}e^{- \delta' \Im{\lambda}}
 \end{align}
 for any $\delta'$ with $0 < \delta' < \delta$ and $\epsilon>0$. In particular $\Xi$ is bounded in each sector near zero and is exponentially decaying at imaginary infinity. 
 Here 
 $$
 \delta = \min_{j \not= k} \mathrm{dist}(\Omega_j,\Omega_k)
 $$ 
 denotes the minimum of the distances between distinct objects.
 The function $\Xi$ has a continuous boundary value on $\R$ and
 if $\lambda>0$ we have 
 $$
   \frac{1}{\pi} \Im \Xi(\lambda) =  -\frac{\rmi}{2\pi} \left( \Xi(\lambda) - \Xi(-\lambda)\right) = -\xi_{\rel}(\lambda).
 $$
Moreover, $\Xi(\lambda)$ can be expressed in terms of boundary layer operators as $\Xi(\lambda) = \log \det \left( Q_\lambda \tilde Q_\lambda^{-1} \right)$
(see \cite[Theorem 1.7]{RTF}). Here  $Q_\lambda$ is the single layer operator for the Helmholtz equation on $\partial \Omega$ and $\tilde Q_\lambda$ is the direct sum of the single layer operators on the components $\partial \Omega_j$. This makes the function accessible both to numerical computation and to explicit estimates.

In this paper we link the asymptotic exponential decay rate of $\Xi$ with wave-trace invariants in the singularity expansion of the Fourier transform $\hat \xi_{\rel}$ of the relative spectral shift function $\xi_{\rel}$. 

This is achieved by establishing a relationship between $\hat \xi_{\rel}$ and $\Xi$. Let $\theta = \chi_{[0,\infty)}$ be the Heaviside step function.
Then, $\Im(\hat \xi_{\rel})$ vanishes near zero and $-4\pi \theta \cdot \Im(\hat \xi_{\rel})$ has a well defined Fourier-Laplace transform. We show in Section \ref{Section:4}, that this Fourier-Laplace transform equals
$\Xi(\lambda)$. This allows to translate properties of the wave-trace, a well studied object, to results for the function $\Xi(\lambda)$. 
In particular the singularity of the wave-trace at $t=2\delta$ determines the decay of $\Xi(\lambda)$ at imaginary infinity. 

To demonstrate this we focus on the simplest case, when the obstacles are strictly convex near the points that have distance $\delta$ to other boundary components.
In this case there exist only finitely many isolated non-degenerate bouncing ball orbits of length $2\delta$ between the obstacles and the corresponding leading wave trace-invariant at $t=2\delta$ can be computed explicitly. This leads to 
the following asymptotic behaviour for $\Xi$ valid uniformly in any sector of the form
$\{ \lambda \in \mathbb{C} \mid \Im(\lambda) > \alpha| \Re(\lambda) | \}, \alpha>0$, namely
$$
 \Xi(\lambda) =- \sum_{j} \frac{1}{|\det(I - P_{\gamma_j})|^{\frac{1}{2}}} e^{2 \rmi \delta \lambda} + o(e^{- 2 \delta \Im{\lambda}}),
$$
where the sum is over  bouncing ball modes of length $2 \delta$ and $P_{\gamma_j}$ is the associated Poincar\'e map. The precise formulation is in Theorem \ref{smoothzerothm} and \ref{CapitalXi thm}. This improves the a priori bounds \eqref{roughbounds} from \cite{RTF} and allows for a geometric interpretation. Without convexity assumptions one has the bound
$$
 \Xi(\lambda) =  O(e^{- 2 \delta' \Im{\lambda}}).
$$
for any $0< \delta'< \delta$ as a consequence of our finite propagation speed estimates uniformly in any sector of the form above.

\subsection{Casimir effect} The quantity $\frac{1}{2\pi}\int_0^\infty \Xi( \rmi \lambda) \der \lambda$ can be interpreted as the Casimir energy between the objects. This can formally be justified by considering the relative trace of the operator as done in \cite{RTF} and quantum field theory considerations. In \cite{FS2021} we provided a full mathematical proof that the Casimir force, as computed from the quantum mechanical stress energy tensor is the same as the variation of the above energy. It also was shown to have the same variation as Zeta regularised quantities. The extension to differential forms will be given in a forthcoming paper \cite{FS2021ef}, which is related to \cite{AS2021,SW2021}. Formal considerations in theoretical physics have been used to justify expansions of the type above. We refer here to  \cite{emig2008casimir,KennethKlich} and in particular \cite{wirzba08} where the relation to scattering theory is claimed. Our results are a further step to a full mathematical justification and show to what extent formal derivations hold and how they need to be interpreted. In particular the relation to mathematical scattering theory is expected to provide further insights.

\subsection{Wave-trace invariants}

For non-compact cases, the wave-trace invariants determine the asymptotic behaviour of the function $\Xi$ in the upper half plane. Conversely the asymptotic behaviour of $\Xi$ can probably be used to compute wave-trace invariants of bouncing ball orbits. The reason is that
$\Xi$ is the determinant of the operator which is expressed entirely in terms of boundary layer operators (\cite{RTF}). Boundary layer operators were used in two dimensions in Zelditch's treatment of the inverse problem for $\mathbb{Z}_2$-symmetric domains \cite{Zelditch2004, Zelditch2009}. A statement for higher dimensions can be found in the work of Hezari and Zelditch \cite{HezariZelditch2010}. The function $\Xi$ may be useful in this context.

\subsection{Sign conventions and notations}

\subsubsection{Function spaces and Fourier transform}

The Fourier transform $\hat f$ of $f \in L^1(\R^d)$ will be defined by $$\hat f(\xi) = \int_{\R^d} f(x) \ee^{-\rmi \, x \cdot \xi} \dd x$$ where $x \cdot \xi$ is the Euclidean inner product on $\R^n$.
We work with the field of complex numbers unless otherwise stated: For example $C^\infty(M)$ denotes the space of complex valued smooth functions on $M$.
Similarly, $C^\infty_0(M)$ denotes the set of complex valued smooth compactly supported functions.

\subsubsection{Convexity and concavity of obstacles}
\label{Convexity and concavity of obstacles}
In this paper, we use the convention of \cite{AnderssonMelrose,Melrose1975} to define the convexity or concavity for a Riemannian manifold $(M,g)$ with boundary $\partial M$.
Let $f: M \to \R$ be a boundary defining function, i.e.
$f=0, \der f \not=0$ on $\partial M$ and  $f >0$ on $M \setminus \partial M$. 
We pull back $f$ to a function on $T^*M$ which we denote by the same letter.
Then $M$ is called strictly (locally geodesically) convex, if for every $(x,\xi) \in T^*M$ with $x \in \partial M$ we have the conclusion
$$
\left( H_g f \right) (x,\xi)=0 \implies \left( H_g H_g f \right)(x,\xi) <0,
$$
where $H_g$ is the generator of the geodesic flow on $T^*M$, i.e. the Hamiltonian flow of the function $\frac{1}{2} g^{-1}(\xi,\xi)$ on $T^*M$.
Similarly, M is called strictly (locally geodesically) concave, if for every $x\in \partial M$
$$
\left( H_g f \right) (x,\xi)=0 \implies \left( H_g H_g f \right)(x,\xi) >0,
$$
These definitions can of course be localised, so it makes sense to say that $M$ is strictly convex/concave locally near a point $x \in \partial M$. If the musical isomorphisms are used to identify $T^*M$ and $TM$ then the vector field $H_g$ gets identified with the geodesic spray. If $g$ is Euclidean, then $H_g f=\nabla f$ is the gradient of $f$ and $H_g H_g f=\operatorname{Hess}(f)$ is the Hessian of $f$. Therefore, the above definitions of convexity and concavity are the standard definitions if $(M,g)$ is Euclidean.

\section{Singularity trace expansion for convex obstacles}

Since the operator $(\Delta_X+1)^{-k} - (\Delta_0 + 1)^{-k}$ is a trace-class operator for all $k>(d-1)/2$ (see \cite{carron1999determinant}), the classical Lifshits-Krein spectral shift function of the pair $(A,B)$ with  $A=(\Delta_X+1)^{-k}$ and $B= (\Delta_0 + 1)^{-k}$ is the unique $L^1$-function $\xi_{AB}(\lambda) \in L^1(\R)$ such that
$$
 \tr(g(A)-g(B)) = -\int_0^\infty \xi_{AB}(\lambda) g'(\lambda) \der \lambda,
$$
for all $g \in C^\infty_0(\R)$. The above identity is known to hold for $g$ in the Besov space $B^{1}_{\infty,1}(\R)$, but is certainly true for $L^1$-functions whose derivative has $L^1$-Fourier transform. In fact, the most general class of admissible functions for $\xi_{AB}$ is the space of operator Lipschitz functions on $\R$ \cite{Peller,Peller1}.
Changing variables this shows that $f(\Delta_X^\frac{1}{2})-f(\Delta_0^\frac{1}{2})$ is trace-class with trace equal to
$$
 \tr( f(\Delta_X^\frac{1}{2})-f(\Delta_0^\frac{1}{2}) ) = -\int_0^\infty \xi(\lambda) f'(\lambda) \der \lambda,
$$
for all $f \in C^1([0,\infty))$ that satisfy 
$$
 f(\lambda) = O(\frac{1}{(1+\lambda^2)^{k}}),\quad f'(\lambda) = O(\frac{\lambda}{(1+\lambda^2)^{k+1}}),
$$ 
for some $k>(d-1)/2$. This is not the most general class of functions for which the above holds, but it will be sufficient for our purposes.
From the above change of variables one obtains $\xi \in L^1(\R,|\lambda|(1+\lambda^2)^{-k-1} \der \lambda)$. 
By the bounds on $\| \mathbi{A}_\lambda\|_1$ one in fact has  $\xi \in L^1(\R,(1+\lambda^2)^{-k-1} \der \lambda)$ and $\xi$ is a piecewise continuous function vanishing at zero.

This implies that $\hat \xi$ is a tempered distribution.
Since $\xi$ is supported in $[0,\infty)$ it is completely determined by its odd part
$\xi_\mathrm{o}(\lambda) = \frac{1}{2} (\xi(\lambda) -\xi(-\lambda) )$. Then 
$$
\hat \xi_\mathrm{o}(t)=\frac{1}{2}\left(\int_0^\infty\xi(\lambda)e^{-\rmi t\lambda}\dd \lambda+\int_{-\infty}^0 -\xi(-\lambda)e^{-\rmi t\lambda}\dd \lambda\right)=- \rmi \int_0^\infty \xi(\lambda) \sin(t \lambda) \dd \lambda .
$$
The distributional wave trace $w \in \mathcal{S}'(\R)$ is formally defined by
$$
w(t):=\tr \left(\cos(t \Delta_X^\frac{1}{2}) - \cos(t \Delta_0^\frac{1}{2})\right).
$$
This formal expression needs to be interpreted in the sense of distributions, i.e. for each test function $\phi \in \mathcal{S}(\R)$ one has 
$$
 \hat \phi_e(\Delta_X^\frac{1}{2})-\hat \phi_e(\Delta_0^\frac{1}{2}) =  \int_\R \phi(t) \left( \cos(t \Delta_X^\frac{1}{2}) - \cos(t \Delta_0^\frac{1}{2})\right) \der t 
$$
is a trace-class operator, and its trace is given by the pairing $(w,\phi)$. Here $\phi_e(t) = \frac{1}{2}(\phi(t) + \phi(-t))$ is the even part of $\phi$ and therefore $\hat \phi_e$ is the cosine transform of $\phi$.
By the Birman-Krein formula we have that $(w, \phi)$ is equal to the pairing of
$ \rmi t \hat \xi_\mathrm{o}$ with $\phi$. As an immediate consequence of the definition of the spectral shift function one obtains
$$
 w(t)=\tr \left( \cos(t \Delta_X^\frac{1}{2}) - \cos(t \Delta_0^\frac{1}{2})\right)= \rmi t \hat \xi_\mathrm{o}(t)= - t \Im \hat \xi(t).
$$

Now we would like to separate the absolutely continuous part $\xi_{ac}$ from the piecewise constant part $\xi_{p}$ of the spectral shift function. Let $\Pi_M:L^2(\R^d)\to L^2(M)$ be the orthogonal projection. Then $ \hat \phi_e(\Delta_X^\frac{1}{2})$ commutes with $\Pi_M$, but $ \hat \phi_e(\Delta_0^\frac{1}{2})$ does not. However, we still have the decomposition
$$
 w = w_{ac} + w_p,
$$
\begin{equation}
\label{wac eqn}
w_{ac}(t) = \tr \left( \Pi_M\left(   \cos(t \Delta_X^\frac{1}{2}) - \cos(t \Delta_0^\frac{1}{2}) \right) \Pi_M\right) - | \Omega | \gamma_{d}(t),
\end{equation}
and
\begin{equation}
\label{wp eqn} 
 w_p(t) = \tr \left((1 - \Pi_M) \cos(t \Delta_X^\frac{1}{2}) (1 - \Pi_M)\right). 
\end{equation}
Here 
$$
| \Omega | \gamma_{d}(t) = \tr \left( (1 - \Pi_M) \left(  \cos(t \Delta_0^\frac{1}{2})  \right)(1 - \Pi_M) \right)
$$
and $\gamma_{d}$ is a homogeneous distribution given by 
\begin{equation*}
\gamma_{d}(t)=\left\{
\begin{aligned}
&\frac{\pi^{\frac{1-d}{2}}\underline{t}^{-d}}{\Gamma(\frac{1-d}{2})}& \text{for even $d$} \\
&(-\pi)^{\frac{1-d}{2}}\frac{(\frac{d-1}{2})!}{(d-1)!}\boldsymbol{\delta}^{d-1}(t)& \text{for odd $d$}
\end{aligned}
\right.
\end{equation*}
where $\underline{t}^{-d}$ is the homogeneous distribution defined in \cite[Section 3.2]{Hormander} and $\boldsymbol{\delta}^{d-1}$ is the $d-1$-th distributional derivative of the delta distribution. Note that $\gamma_{d}(t)$ is the cosine Fourier transform of $(d C_d)\lambda_+^{d-1}$ \cite{Safarov}.

One can also consider the Cauchy evolution operator $U_X(t)$ which acts on the Hilbert space $L^2(X) \oplus L^2(X)$ and is given by
$$
 U_X(t) = \left( \begin{matrix}  \cos(t \Delta_X^\frac{1}{2}) & \Delta_X^{-\frac{1}{2}} \sin(t \Delta_X^\frac{1}{2}) \\  -\Delta_X^{\frac{1}{2}} \sin(t \Delta_X^\frac{1}{2}) &  \cos(t \Delta_X^\frac{1}{2}) \end{matrix}\right),
$$
where $\Delta_X^{-\frac{1}{2}} \sin(t \Delta_X^\frac{1}{2})$ is defined by functional calculus with respect to the function 
$g(x)=x^{-\frac{1}{2}} \sin(x^\frac{1}{2} t)$, which is entire in $x$. This operator has finite propagation speed in the sense that its distributional kernel is supported away from the set
$\{ (x,y) \mid \mathrm{dist}(x,y) > t \}$. We similarly define $U_{X_j}$ and $U_0$. The corresponding distribution trace
$$
 u(t) = \tr(U_X(t) - U_0(t))
$$
then equals $2 w(t)$. 
\begin{rem}
It is sometimes more natural to consider the operator $U_X(t)$ as an operator in $H^1(X) \oplus L^2(X)$ and thus define its distributional trace in that space. Since for any smooth compactly supported test function $\phi \in C^\infty_0(\R)$
the operator 
$
 \int_\R U_X(t) \phi(t) \der t
$
has smooth compactly supported integral kernel, its trace in any Sobolev space equals the integral over the diagonal. Thus, the distributional traces one obtains are independent of the choice of Sobolev space in the definition of the trace.
\end{rem}

The properties of $u$ (and hence $w$) have been subject to extensive investigation in various settings. This started with the work of Lax-Phillips \cite{LaxPhillips,LaxPhillips1978} in odd dimensions. Their method is also known as the Lax-Phillips semigroup construction \cite[Chapter 9]{Taylorbook2}. The trace of the Lax-Phillips semigroup can be expressed in terms of the scattering operator. By Lidskii's theorem, the trace can also be written as a sum of the Fourier transform of the test function over the scattering frequencies (also known as resonances or scattering poles). It turns out that the trace of the Lax-Phillips semigroup is equal to the one of $u(t)$ \cite[Chapter 9]{Taylorbook2}. The trace-class property of $u$ on $(0,\infty)$ was shown by Bardos-Guillot-Ralston in \cite{MR668585} using the Birman-Krein formula (see also Melrose in \cite{Melrose1982} for potential scattering problems). By applying Ivrii's work on the second Weyl coefficient \cite{Ivrii1980}, Melrose extended the trace formula to include $t=0$ for compact obstacles in \cite{Melrose1988}. There is a rich mathematical literature on estimates of the number of scattering poles in various settings based on trace-formulae, for instance, \cite{Melrose1983,Melrose1988,Petkov,SZ1991,SZ1993,SZ1994}.

The relation to geometry is facilitated by
expressing the singular part of the wave-trace  in terms of a sum of Lagrangian distributions with supports at the lengths of periodic trajectories. This is sometimes referred to as the Poisson summation formula for the wave trace. Thanks to the work of Chazarain \cite{Chazarain}, Colin de Verdi\`{e}re \cite{Colin1,Colin2}, Duistermaat \cite{DuistermaatGuillemin}, Guillemin-Melrose \cite{GuilleminMelrose}, and Andersson-Melrose  \cite{AnderssonMelrose}, the Poisson summation formula for the wave trace was derived for compact manifolds with or without boundary. We also refer to \cite{SVbook} for a very detailed treatment containing the case of manifolds with boundary.
The standard assumptions  for these results is strict geodesic concavity or convexity of the boundary. In the non-compact situation of obstacle scattering of finitely many strictly convex compact objects the Poisson summation formula is due to Bardos, Guillot, and Ralston \cite{MR668585}. We also refer to 
Petkov and Popov \cite{PetkovPopov} for an overview and further results for non-trapping boundaries.

\begin{theorem}
	\label{length spectrum thm}
 Suppose $\Omega_j$ is strictly convex for $1\le j \le N$. Then the singular support of the distribution $w_{ac}$ is contained in the set 
 $$\{0\} \cup \{ t\in \R\, \big|\, \text{$|t|$ is the length of a periodic trajectory (closed broken geodesic) in $M$} \}.$$
\end{theorem}

This theorem can be found in somewhat different language in \cite{AnderssonMelrose,GuilleminMelrose,PetkovPopov} and is essentially Theorem 5 in \cite{MR668585}. It was first proved in a manifold setting by Melrose and Andersson \cite{AnderssonMelrose} for compact manifolds. The theorem in the non-compact context can  be reduced to the compact case using finite propagation speed arguments. Since such arguments are important in our approach we sketch it here for the sake of completeness.

\begin{proof}
Since $w_{ac}$ is even it is sufficient to prove this for $w_{ac}$ restricted to an arbitrary interval $(0,T)$. We can therefore assume without loss of generality that all test functions are supported in $(0,T)$.
Denote by  $\tilde w_M(t,x,y)$ and $\tilde w_0(t,x,y)$ the distributional kernels for $\cos(t \Delta_M^\frac{1}{2})$ and $\cos(t \Delta_0^\frac{1}{2})$ respectively.
Let $\phi \in C^\infty_0((0,T))$ be an arbitrary test function. Then the operator
$$
 \int_\R \left(  \cos(t \Delta_M^\frac{1}{2}) - \cos(t \Delta_0^\frac{1}{2}) \right) \phi(t) \der t
$$
has integral kernel that is smooth on $\overline{M} \times \overline{M}$. In fact, $\tilde w_M(t,x,y)$ and $\tilde w_0(t,x,y)$ are distributions on $(0,T)$ taking values in $C^\infty(\overline{M} \times \overline{M})$ (see\cite{AnderssonMelrose},\cite[Section 1]{DuistermaatGuillemin},\cite[Section 3]{DuistermaatHomander} and \cite[Section 4]{PetkovPopov}) and we will use this as a convenient notation. For example the diagonal values 
$\tilde w_M(t,x,x)$ and $\tilde w_0(t,x,x)$ make sense as distributions in the $t$-variable.
In fact, by finite propagation speed its kernel is compactly supported in $\overline{M} \times \overline{M}$.
By Mercer's theorem we have in the sense of distributions
\begin{equation*}
\label{wM eqn}
w_M(t) := \tr \left( \Pi_M\left(  \cos(t \Delta_X^\frac{1}{2}) - \cos(t \Delta_0^\frac{1}{2}) \right) \Pi_M\right) = \int_M w_M(t,x,x) - w_0(t,x,x) \der x.
\end{equation*}
The support of $w_M(t,x,x)-w_0(t,x,x)$ is compact and contained in a ball $B_R(0)$ of radius $R>T>0$. We choose $R$ so large that the distance from the boundary of the ball to $\Omega$ is larger than $T$.
Again using finite propagation speed and the implied support properties of the wave-kernels we can modify $M$ outside this ball without changing the integral. This is done by gluing a large $d$-dimensional hemisphere onto the boundary of $B_R(0)$ in such a way that no additional length spectrum in $(0,T)$ is introduced.
The precise gluing construction can be found in \cite[Section 10]{S1997} (also in \cite{SZ1993}). 
In this way we obtain a compact manifold $\tilde M$ with boundary $\partial \Omega$ and a closed manifold $\tilde M_0$ such that $\tilde M =\tilde M_0 \setminus \Omega$.
Since the closed manifold $\tilde M_0$ was constructed from a large ball in $\R^d$ by gluing a large hemisphere the length spectrum of this manifold does not contain elements in $(0,T)$.
We have as a distribution on $(0,T)$ the equality
\begin{equation*}
w_M(t) = \int_{\tilde M} w_{\tilde M}(t,x,x) \der x -  \int_{\tilde M_0} w_{\tilde M_0}(t,x,x) \der x +  \int_{\Omega} w_{\tilde M_0}(t,x,x) \der x,
\end{equation*}
 where $w_{\tilde M}(t)=\tr (\cos(t \Delta_{\tilde M}^\frac{1}{2}))$ and $w_{\tilde M_0}(t)=\tr (\cos(t \Delta_{\tilde M_0}^\frac{1}{2}))$. The second term has no singularity in $(0,T)$ since $\tilde M_0$ is a closed manifold and the length spectrum does not intersect $(0,T)$, by \cite{DuistermaatGuillemin}. The third term is  $| \Omega | \gamma_{d}(t)$ when restricted to $(0,T)$ and also does not have any singularities.
The singularities of $w_{\tilde M}(t)$ were studied in \cite{AnderssonMelrose}. In particular, Theorem  8.9 in \cite{AnderssonMelrose} implies $\Sing (w_{\tilde M}) \subset \mathcal{L}(M)$, where $\mathcal{L}(M)$ is the minimal length spectrum of $M$ and it is defined as
\begin{multline*}
\mathcal{L}(M)=\{0\} \cup \{ t\in \R\, \big|\, \text{$|t|$ is the length of a periodic trajectory (closed broken geodesic)}
\\
\text{ in $M$ or of a closed boundary geodesic in a strictly convex component.}
\}
\end{multline*}
Since $\Omega_j$ is strictly convex with respect to the interior part, it is strictly concave with respect to $\tilde M$ (the exterior part). Therefore, there are no gliding rays in $\tilde M$ and the theorem follows.
\end{proof}

\section{Trace singularity expansion for the relative spectral shift function}

In this section we assume throughout that number of connected components $N$ is at least two. We will study the singularities of $\hat \xi_{\rel}$ and the relative distributional wave-trace
\begin{equation}
\label{wrel eqn}
 w_\rel(t) = w(t) - \sum_{j=1}^N w_j(t).
\end{equation}
Here $w_j(t) = \tr \left( \cos(t \Delta_{X_j}^\frac{1}{2}) - \cos(t \Delta_0^\frac{1}{2})\right)$ corresponds to $w(t)$ in the configuration where only the $j$-th obstacle $\Omega_j$ is present. 

We start by preparing some observations about finite propagation speed which hold independent of convexity assumptions.

\begin{proposition}
	\label{1obstacledifference}
 Let $\tilde u(t,x,y)$ be the distributional kernel of $U_X-U_0$. Then, for $(t,x,y)$ to be in the support of 
 $\tilde u$ it is a necessary condition that there exists a piecewise linear continuous path $\gamma: [0,L] \to \R^d$ of length 
 $L \leq |t|$ such that $\gamma(0)=y$, $\gamma(L)=x$ and $\gamma(s) \in \partial \Omega$ for some $s \in [0,L]$.
\end{proposition}
\begin{proof}
We first prove the statement for $t\ge 0$. For $g = \left( \begin{matrix} g_1\\ g_2 \end{matrix} \right) \in C^\infty(\R \times X, \C^2)$ we write $\mathcal{D} g := \partial_t g_2 -g_1$. If $f = \left( \begin{matrix} f_1\\ f_2 \end{matrix} \right) \in C^\infty_0(X, \C^2)$ then $g(t,\cdot)=U_X(t) f$ solves the system
$\Box g =0, \mathcal{D} g=0$ with initial conditions $g(0,x)= f(x)$. The function $g$ is the unique solution of this system satisfying the boundary conditions.

Let $A\subset \R^d$ and define $A_r:=\{x\in\R^d: \dist(x,A)\le r \}$. We also set $A_r=\emptyset$ for $r<0$ and $\emptyset_r=\emptyset$ for $r\in \R$. 
We fix $(x,y) \in X \times X$ and
consider a $\C^2$-valued test function $f$ supported in an $\epsilon$-ball $B_\epsilon(y)$. Moreover, let $\eta$ be a test function which is supported in $(\partial \Omega)_\epsilon$, i.e. a small $\epsilon$-tubular neighborhood of $\partial \Omega$ and $\eta=1$ on $(\partial \Omega)_{\frac{\epsilon}{2}}$. Let $\chi = 1- \eta$.
Then $g=\chi  U_0(t)f$ is a $\C^2$-valued solution of the inhomogeneous wave equation
\begin{equation}
\label{inhomo wave eqn}
 \Box g(t) = (\partial_t^2 + \Delta) g(t) =  [\Box, \chi] U_0(t)f = h(t), \quad \mathcal{D} g=0.
\end{equation}

This implies $\mathcal{D} h=0$. 
Since $h(t) =  -[\Box, \eta] U_0(t)f$, one knows that $h(t)$ is supported  in $(\partial \Omega)_\epsilon\cap B_{t+\epsilon}(y)$. This means that the support of $h$ contains points $(t,z)$
only if there is a linear path starting in $B_\epsilon(y)$ ending at $z$ in $(\partial \Omega)_\epsilon\cap B_{t+\epsilon}(y)$ with a length of $L_1 \leq t+2\epsilon$. Let $G_{\mathrm{ret}}$ be the forward propagator obtained from $U_X$, i.e. $G_{\mathrm{ret}}(s)$ is supported at $s\ge 0$ and it is given by
$$
G_{\mathrm{ret}}(s)= \theta(s) \left( \begin{matrix}  \Delta_X^{-\frac{1}{2}} \sin(s \Delta_X^\frac{1}{2}) & 0 \\  0 &  \Delta_X^{-\frac{1}{2}} \sin(s \Delta_X^\frac{1}{2}) \end{matrix}\right).
$$
As before $\theta$ denotes the Heaviside step function.
Since $g(t)$ satisfies the boundary condition for $\Delta_X$ and the inhomogeneous wave equation \eqref{inhomo wave eqn} with $g(0)=f$, we have 
$$
g(t) =  U_X(t) f + \int_{0}^t G_{\mathrm{ret}}(t-t') h(t') \der t'.
$$
For sufficiently small $\epsilon>0$ let $\varphi$ be a smooth cutoff function supported in $B_{\epsilon}(x)$. Then,
$$
\varphi\left(U_X(t)- U_0(t)\right) f = - \int_{0}^t\varphi G_{\mathrm{ret}}(t-t') h(t') \der t'.
$$
As a consequence of energy estimates with boundary conditions, $G_{\mathrm{ret}}(t)$ has the finite propagation speed property in the sense that its distributional kernel is supported in $\{(x,y) \mid \mathrm{dist}(x,y) \leq t\}$. Therefore, $G_{\mathrm{ret}}(t-t') h(t') $ is supported in $((\partial \Omega)_\epsilon\cap B_{t'+\epsilon}(y))_{t-t'}$. In order for the support of $\varphi G_{\mathrm{ret}}(t-t') h(t')$ to be non-empty, we must have $\supp(\varphi) \cap ((\partial \Omega)_\epsilon\cap B_{t'+\epsilon}(y))_{t-t'}\ne \emptyset$, which holds only if $B_\epsilon(x) \cap ((\partial \Omega)_\epsilon)_{t-t'}\ne \emptyset$. This means that there exits $z\in(\partial \Omega)_\epsilon$ such that $B_\epsilon(x) \cap B_{t-t'}(z)\ne \emptyset$. Hence, there must be another linear path of length $L_2$ starting in $z \in(\partial \Omega)_\epsilon$ ending in $B_{\epsilon}(x)\cap B_{t-t'}(z)$. Therefore, we conclude that $L_2\le t-t'+2\epsilon$. Taking into account of the constraint on $L_1$, we have
\begin{equation*}
\left\{
\begin{aligned}
&L_1\leq t'+2\epsilon\\
&L_2\le t-t'+2\epsilon
\end{aligned}
\right.
\implies L=L_1+L_2\leq t+4\epsilon,
\end{equation*}
where $L$ is the total length of the piecewise linear continuous path. If $\gamma: [0,L] \to \R^d$ parameterises the piecewise linear continuous path, we would have $\gamma(0)\in B_\epsilon(y)$, $\gamma(s)\in (\partial \Omega)_\epsilon$ for some $s\in [0,L]$ and $\gamma(L)\in B_\epsilon(x)$. Finally, $\epsilon>0$ can be chosen arbitrarily small and the statement for $t\ge 0$ follows. For the negative time, we consider $v(t,x,y)=\tilde{u}(-t,x,y)$, which corresponds to the kernel of $U_X(-t)-U_0(-t)$ for $t\ge 0$ and the statement of the existence of a piecewise linear continuous path for $v$ with $t\ge 0$ follows the same as the above construction for $\tilde{u}$ with $t\ge 0$. Hence our statement holds for $\tilde{u}$ with $L\le |t|$.
\end{proof}

This means essentially signals starting at $x$ propagate initially with respect to $U_0$ until the wave hits the object, then the effect of the object will be additional reflected waves that also travel at finite speed and need the additional time to reach the point $y$. Essentially the same proof shows the following.
\begin{proposition}
	\label{1obstacledeformation}
	Suppose that $\Omega,\Omega'$ are two different collections of obstacles and let $X$ and $X'$ be the complements of $\partial \Omega$ and $\partial \Omega'$, respectively.
  Then, for $(t,x,y)$ to be in the support of the distributional kernel of $U_{X'}-U_{X}$ it is a necessary condition that there exists a piecewise linear continuous path 
  $\gamma: [0,L] \to \R^d$ of length 
 $L \leq |t|$ such that $\gamma(0)=y$, $\gamma(L)=x$ and $\gamma(s) \in (\partial \Omega \setminus \partial\Omega')\cup (\partial\Omega' \setminus \partial\Omega)$ for some $s \in [0,L]$.
\end{proposition}

The operator $U_\rel$ has a similar property, but the piecewise linear path in this case has to travel via at least two objects to pick up an effect. This is made precise in the theorem below.

\begin{theorem}
\label{Nobstacledifference}
Let $\tilde u_\rel(t,x,y)$ be the distributional kernel of $$U_\rel=U_X - \left(\sum_j U_{X_j}-(N-1)U_0\right).$$ Then, for $(t,x,y)$ to be in the support of $\tilde u_\rel$ it is a necessary condition that there exists a piecewise linear continuous path $\gamma: [0,L] \to \R^d$ of length $L \leq |t|$ such that $\gamma(0)=x$, $\gamma(L)=y$ and there exist $j_1 \not= j_2$ and $s_1,s_2 \in [0,L]$ so that $\gamma(s_1) \in \partial \Omega_{j_1}$ and $\gamma(s_2) \in \partial \Omega_{j_2}$.
\end{theorem}
\begin{proof}
We continue using the same notation as in the proof of Proposition \ref{1obstacledifference}. As before the statement also follows from a finite propagation speed consideration. Of course we can assume without loss of generality that $N>1$. Fix $y \in X$ and let $f \in C^\infty_0(X, \C^2)$ be supported in $B_\epsilon(y)$ for some small $\epsilon>0$. 
Let $\calO_j=\bigcup_{k\not=j} \partial \Omega_k=\partial \Omega \setminus \partial \Omega_j$. Furthermore, let $\eta_j$ be a cutoff function supported in $(\calO_j)_\epsilon$ such that $\eta_j=1$ on $(\partial \Omega_k)_{\frac{\epsilon}{2}}$ whenever $k \not=j$. Now define $\eta = \frac{1}{N-1} \sum_{j=1}^N \eta_j$ and observe that $\eta=1$ on $(\partial \Omega)_{\frac{\epsilon}{2}}$.
Next define $\chi = 1 - \eta$ and $\chi_j = 1 - \eta_j$. We consider the function 
$$
g(t) =  \left( (\sum_{j=1}^N \chi_j U_{X_j}(t)) - (N-1) \chi U_0 \right) f.
$$
Then for each $t$, $g(t)$ is a smooth compactly supported function. Moreover, $g(t)$ satisfies the inhomogeneous wave equation
$$
(\partial_t^2 + \Delta) g(t,x) = h(t,x)
$$
on $X$ with initial conditions $g(0,x) = f(0,x)$ and it also satisfies the boundary conditions. Moreover, $\mathcal{D} g=0$ and hence $\mathcal{D} h=0$.
We compute 
\begin{gather*}
h(t) =  (\sum_{j=1}^N [\Box, \chi_j] U_{X_j}(t)) - (N-1) [\Box, \chi]  U_0  f \\
= -\left((\sum_{j=1}^N [\Box, \eta_j] U_{X_j}(t) - (N-1) [\Box, \eta]  U_0)\right) f = - \left( \sum_{j=1}^N [\Box, \eta_j] (U_{X_j}(t)-U_0(t) )\right)  f
\end{gather*}
Next observe that
$$
U_X(t) f = g(t) - \int_0^t G_{\mathrm{ret}}(t-t') h(t') \der t',
$$
since $U_X(t) f$ solves the initial value problem with Dirichlet boundary conditions on $\partial \Omega$. Let $\varphi$ be a smooth cutoff function supported in $B_{\epsilon}(x)$. Then we know, for sufficiently small $\epsilon>0$,
$$
\varphi  U_\rel(t) f=\varphi  U_X(t) f- \varphi  g(t)= -\int_0^t \varphi G_{\mathrm{ret}}(t-t') h(t') \der t'.
$$
That is
$$
\varphi  U_\rel(t) f= -\int_0^t \varphi G_{\mathrm{ret}}(t-t')  \left( \sum_{j=1}^N [\Box, \eta_j] (U_{X_j}(t')-U_0(t') )\right)f \der t'.
$$
Observe that $[\Box, \eta_j]$ is supported in $(\calO_j)_\epsilon$. By the analysis of Proposition \ref{1obstacledifference}, we know that the support of $[\Box, \eta_j] (U_{X_j}(t')-U_0(t') )f$ is non-empty only if there is a piecewise linear continuous path of length $L_1+L_2$ starting in $B_\epsilon(y)$, passing through $(\partial \Omega_j)_\epsilon$ and ending in $(\calO_j)_\epsilon$. Repeating the same arguments, we know that the support of $\varphi G_{\mathrm{ret}}(t-t')  \left( \sum_{j=1}^N [\Box, \eta_j] (U_{X_j}(t')-U_0(t') )\right)$ is non-empty only if there is another linear path of length $L_3$ such that it connects $(\calO_j)_\epsilon$ and $B_{\epsilon}(x)$. As in Proposition \ref{1obstacledifference}, we have $t\ge L=L_1+L_2+L_3-6\epsilon$. The statement for $t\ge0$ now follows by the fact that one can choose $\epsilon>0$ arbitrarily small. A similar argument applies to the case $t\le0$.
\end{proof}

Essentially the same proof also shows another manifestation of finite propagation speed.

\begin{theorem}
\label{Nobstacledeformation}
Suppose that $\Omega = \Omega_1 \cup \ldots \cup \Omega_N$ and $\Omega' = \Omega_1' \cup \Omega_2 \cup \ldots \Omega_N$ are 
two collections of obstacles.
Let $U_\rel$ and ${U'}_\rel$ be the corresponding relative operators.
Then, for $(t,x,y)$ to be in the support of the distributional kernel of ${U'}_\rel - U_\rel$ it is a necessary condition that there exists a piecewise linear path
 $\gamma: [0,L] \to \R^d$ of length $L \leq |t|$ such that $\{\gamma(L),\gamma(0)\}=\{x,y\}$ and there exist $j \not=1$ and $s_1,s_2 \in [0,L]$ so that 
 $\gamma(s_1) \in (\partial\Omega_{1} \setminus \partial\Omega'_{1}) \cup  (\partial\Omega'_{1} \setminus \partial\Omega_{1})$ and $\gamma(s_2) \in \partial\Omega_{j} $.
 \end{theorem}
\begin{proof} 
For brevity denote $\partial\calO_2 \cup \ldots \cup \partial\calO_N$ by $\partial\calO_c$. We fix $\epsilon>0$ sufficiently small and choose a cutoff function
$\eta_c \in C^\infty(\R^d)$ supported in $(\partial \Omega_c)_\epsilon$ that equals one near $( \partial\Omega_c)_\frac{\epsilon}{2}$. As before define $\chi_c=1-\eta_c$. We fix $f \in C^\infty_0(B_\epsilon(y),\C^2)$ and $\varphi \in C^\infty_0(B_\epsilon(x))$.

Now choose $\eta_0\in C^\infty_0(\R^d)$ with support in $\left((\partial\Omega_{1} \setminus \partial\Omega'_{1}) \cup  (\partial\Omega'_{1} \setminus \partial\Omega_{1})\right)_\epsilon$ such that $\eta_0=1$ on $\left((\partial\Omega_{1} \setminus \partial\Omega'_{1}) \cup  (\partial\Omega'_{1} \setminus \partial\Omega_{1})\right)_{\frac{\epsilon}{2}}$. Then we set $\chi_0=1-\eta_0$ and $\chi_c'=\chi_0 \chi_c=1-\eta_0-\eta_c$. 
We need to analyse under which conditions the distribution
$$
\varphi \left( U_\rel-U_\rel' \right)f=\varphi \left( \left(U_X-U_{X'}\right)-\left(U_{X_1}-U_{X'_1}\right) \right)f=\varphi U_X(t) f- \varphi g(t).
$$
is non-zero, where $X_1'=\R^d\backslash\partial \Omega_{1}'$ and
$
g=\chi_0 U_{X'}+ \chi_c U_{X_{1}}- \chi_c' U_{X_{1}'}$. The function $g$ satisfies boundary conditions on $X$
and solves the inhomogeneous wave equation
$\Box g= h$, where
\begin{align}
h =  -\left([\Box,\eta_0]\left(U_{X'}-U_{X'_1}\right)+[\Box,\eta_c]\left(U_{X_1}-U_{X'_1}\right)\right)f.\label{h of Nobstacledeformation}
\end{align}
 Then we have
$$
 g(t) =  U_X(t) f+\int_0^t  G_{X,\mathrm{ret}}(t-t') h(t') \der t',
$$
where $G_{X,\mathrm{ret}}$ is the forward propagator obtained from $X$. Together with equation \eqref{h of Nobstacledeformation}, one deduces
\begin{multline*}
\varphi \left( U_\rel-U_\rel' \right)f=- \varphi  \int_0^t  G_{X,\mathrm{ret}}(t-t') h(t') \der t'
\\
=\int_0^t  \varphi G_{X,\mathrm{ret}}(t-t')\left([\Box,\eta_0]\left(U_{X'}-U_{X'_1}\right)+ [\Box,\eta_c]\left(U_{X_1}-U_{X'_1}\right)\right)(t')f \der t'.
\end{multline*}
If this is non-zero we must have that either $\varphi G_{X,\mathrm{ret}}(t-t')[\Box,\eta_0]\left(U_{X'}-U_{X'_1}\right)(t') f$ is nonzero for some $0\le t'\le t$, or that $\varphi G_{X,\mathrm{ret}}(t-t')  [\Box,\eta_1]\left(U_{X_1}-U_{X'_1}\right)(t')f$ is nonzero for some $0\le t'\le t$. Suppose that the first term is non-zero.
We note that $[\Box,\eta_0]$ is supported in $\left((\partial\Omega_{1} \setminus \partial\Omega'_{1}) \cup  (\partial\Omega'_{1} \setminus \partial\Omega_{1})\right)_\epsilon$ and the symmetric difference of $\partial X'$ and $\partial \Omega'_1$ is $\partial \Omega_c$. Applying Proposition \ref{1obstacledeformation} to $U_{X'}$ and $U_{X'_1}$, we conclude that there exits a piecewise linear continuous path starting from $B_\epsilon(y)$ to $(\partial \Omega_c)_\epsilon$ ($\epsilon$-neighbourhood of boundaries of all the other obstacles) and then ends in $\left((\partial\Omega_{1} \setminus \partial\Omega'_{1}) \cup  (\partial\Omega'_{1} \setminus \partial\Omega_{1})\right)_\epsilon$. A similar argument applies to the second term.
\end{proof}

Recall that $\delta$ is the minimal distance between two objects, i.e.
$$
\delta = \inf \{ \mathrm{dist}(x,y) \mid x \in \partial \Omega_j, y \in \partial \Omega_k, j \not=k \}.
$$
The behaviour of $w_\rel$ around origin is given in the following corollary. 
\begin{corollary}
 The distribution $w_\rel$ is supported away from $(-2 \delta, 2 \delta)$.
\end{corollary}
\begin{proof}
In the proof of Theorem \ref{Nobstacledifference} we know that if $(t,x,y)$ is in the support of $\tilde u_\rel$, then there exists a piecewise linear continuous path of length $L\le |t|$ that is reflected by two different obstacles, where $y$ and $x$ are starting and ending points respectively. Since $w_\rel(t)$ is the trace of $\frac{1}{2}U_\rel(t)$, we know that if $t$ is in the support of $w_\rel(t)$, then there is a piecewise linear continuous closed path that intersects two different obstacles. Hence by the triangle inequality, we conclude that $L\ge 2 \delta$, which also implies $|t|\ge 2 \delta$.
\end{proof}
Equations \eqref{wac eqn}, \eqref{wp eqn} and \eqref{wrel eqn} imply
$$
w_\rel=\tr \left( \Pi_M  \left[\left(\cos(t \Delta_X^\frac{1}{2}) - \cos(t \Delta_0^\frac{1}{2}) \right)
-\sum_{j=1}^N   \left( \cos(t \Delta_{X_j}^\frac{1}{2}) - \cos(t \Delta_0^\frac{1}{2})   \right)\right] \Pi_M\right),
$$
which means the singular behaviour of $w_\rel$ boils down to the study of singular supports of $w_M$ in \eqref{wM eqn} with different obstacle configurations. Therefore, we could use Theorem \ref{length spectrum thm} to study the singular support of the distribution $\hat \xi_\rel$. The contribution of an isolated non-degenerate periodic billiard trajectory can be computed via the Gutzwiller-Duistermaat-Guillemin formula \cite{DuistermaatGuillemin,Gutzwiller}. 

To simplify the discussion we impose the condition that the obstacles are locally strictly convex near points that have distance equal to $\delta$ from the other obstacles.
This will guarantee that there is a finite number of isolated non-degenerate bouncing ball orbits of length $2 \delta$ between the different obstacles and the Maslov index vanishes  (see Theorem \ref{smoothzerothm}).
Under this hypothesis one can easily compute the leading singularity of $\hat \xi_{\rel}$. 

A bouncing ball orbit is a $2$-link periodic trajectory of the billiard flow. The existence of a non-degenerate bouncing ball orbit plays an important role in Zelditch's work on inverse spectral problems for analytic domains \cite{Zelditch2004,Zelditch2009}. In general, shortest periodic billiard trajectories in a smooth domain are not necessarily bouncing ball orbits (see Ghomi \cite{Ghomi} for a discussion and geometric conditions that ensure this). In our setting this does however not cause a problem.

We make this now precise by introducing the set $\mathcal{B}_{\delta}$ as
$$
\mathcal{B}_{\delta} = \{ (x,y) \in \partial\Omega \times \partial\Omega  \mid  \mathrm{dist}(x,y) = \delta, (x,y) \in \partial\Omega_i \times \partial\Omega_j \textrm{ with }i\ne j \}.
$$
The set $\mathcal{B}_{\delta}$ is symmetric and we define $\mathcal{B}_{\partial \Omega,\delta}$ to be the projection
of $\mathcal{B}_{\delta}$ on the first factor, i.e. $\mathcal{B}_{\partial \Omega,\delta} = \{ x \in \partial \Omega \mid \exists y \in \partial \Omega, (x,y) \in \mathcal{B}_{\delta}\}$.

We have the following elementary proposition.

\begin{proposition}
	\label{bouncing ball orbit}
Let $\delta>0$, as before, be the minimal distance between the disconnected components. For two points $q_1 \in \partial \Omega_i$,
$q_2 \in \partial \Omega_j$ with $i \not=j$ and $\mathrm{dist}(q_1,q_2)=\delta$ denote by $\overline{q_1q_2}$ the linear path connecting them.
Then $\overline{q_1q_2}$ is a bouncing ball orbit with period $2\delta$.
\end{proposition}

\begin{proof}
The set $\{ (q_1,q_2) \in \partial\Omega \times \partial\Omega  \mid (q_1,q_2) \in \partial\Omega_i \times \partial\Omega_j \textrm{ with }i\ne j \}$ is compact and therefore $\mathcal{B}_{\delta}$ is non-empty and for all 
$(q_1,q_2) \in \mathcal{B}_{\delta}$ there exists a straight line in $\R^d$ of length $\delta$ connecting them. This straight line does not intersect any other points of $\partial \Omega$ since that would give a pair points of distance smaller than $\delta$. Therefore this straight line $\overline{q_1q_2}$ is in $M$ and we can restrict to the case when only two compact obstacles are present, i.e. $\Omega=\Omega_1 \cup \Omega_2$.  If $\phi_j$ is a boundary defining function for $\Omega_j$ then
a length minimising straight line satisfies $\der [(q_1-q_2)^2 - 2\alpha \phi_1(q_1) - 2\beta \phi_1(q_2) ]=0$ for some Lagrange multipliers $\alpha,\beta$. This leads to
$q_1-q_2 = \alpha \der \phi_1(q_1)$ and $q_1-q_2 = -\beta \der \phi_1(q_1)$. Hence, $q_1-q_2$ is normal to the tangent space of $T_{q_j} \partial \Omega_j$ and therefore this is a bouncing ball orbit.
\end{proof}

If the $\Omega$ is strictly convex in a neighborhood of $\mathcal{B}_{\partial \Omega,\delta}$ then $\mathcal{B}_{\partial \Omega,\delta}$ is actually a discrete set of points consisting then of the reflection points of bouncing ball orbits between objects.

Since the relative spectral shift function, $\xi_\rel(\lambda)$, only makes sense for at least two obstacles, we now assume that $\Omega$ has at least two compact connected components. One of immediate consequence of Proposition \ref{bouncing ball orbit} is the following theorem.

\begin{theorem}
	\label{smoothzerothm}
The distribution $t\hat \xi_{\rel}(t)$ is real-analytic in $(-2\delta,2\delta)$ and its imaginary part vanishes in $(-2\delta,2\delta)$.
If $\Omega$ is locally strictly convex near $\mathcal{B}_{\partial \Omega,\delta}$ then there is an isolated singularity of $t\hat \xi_{\rel}(t)$ at $2\delta$ of the form
\begin{equation}
\label{firstsingularity}
 \sum_{\gamma_j}\frac{\delta}{\pi |\det(I-P_{\gamma_j})|^{\frac{1}{2}}}(t-2\delta+\rmi 0)^{-1}+L^1_{\loc},
\end{equation}
in $\mathcal{D}'(\R)/L^1_{\loc}(\R)$.
Here $\gamma_j$ are the shortest periodic billiard trajectory between the objects. Here $P_{\gamma_j}$ is the linear Poincar\'{e} map of $\gamma_j$. 
\end{theorem}

\begin{proof}
Recall that
$$
 w_\rel(t) = -t \Im{\hat \xi_\rel}(t).
$$
Since $w_\rel$ vanishes on $(-2 \delta,2 \delta)$ the distribution $g$ obtained by restricting $t\hat \xi_\rel(t)$ to  $(-2 \delta,2 \delta)$ is real valued. 
Since the analytic wave-front set of the complex conjugate of a distribution is obtained by reflection its wavefront set about the origin in the fibres of the cotangent bundle we have that $\mathrm{WF}_A(g)$ is invariant with respect to this reflection. On the other hand $\hat \xi_\rel$
is the boundary value of a function that is analytic in the upper half plane. Hence, $\mathrm{WF}_A(g)$ is one-sided.
It follows that the analytic wavefront set of $g$ is empty and $g$ is real analytic.
We will now use that $\Omega$ is locally strictly convex near $\mathcal{B}_{\partial \Omega,\delta}$. Using Theorem \ref{Nobstacledeformation} we can change the obstacles away from the set $\mathcal{B}_{\partial \Omega,\delta}$ without changing the relative wave trace in a neighborhood of the interval $(0,2 \delta)$. It is straightfoward to see that the obstacles can modified in this way into strictly convex ones. We can therefore assume without loss of generality that the obstacles are strictly convex.

In general, the singularities of $\Im(t\hat \xi_\rel(t))$ are contained in the length spectrum of $M$ as described in Theorem \ref{length spectrum thm}. 
Hence, the first non-trivial singularity can only appear at $t=2\delta$. 
By convexity the set of $2\delta$-periodic billiard trajectories consists of simple non-degenerate billiard trajectories with zero Maslov index. 
 By Proposition \ref{bouncing ball orbit}, these trajectories are also bouncing ball orbits. Both statements can be found in \cite{MR668585} and we therefore only briefly show the computation involved.

If $T\in \Sing(w_\rel)$ and only closed simple billiard trajectories are of $T$-period, then one concludes from Duistermaat, Guillemin and Melrose's work \cite{DuistermaatGuillemin,GuilleminMelrose} that the singularity of $w_\rel$ at $t=T$ is given by the real part of
\begin{equation}
\label{GMtypesingularity}
\left[\sum_{\gamma_j}\rmi^{\sigma_j}(-1)^{N_{j}}\frac{T_{j}^{\#}}{2\pi |\det(I-P_{\gamma_j})|^{\frac{1}{2}}}\frac{\rmi}{t-T+\rmi 0}\right]+L^1_{\loc},
\end{equation}
where $T_{j}^{\#}$ is the primitive period of $\gamma_j$, $N_{j}$ is the number of reflections in $\gamma_j$ and $\sigma_j$ is the Maslov index associated with $\gamma_j$ (see  \cite[Theorem 2]{GuilleminMelrose}).

Since the number of reflections is two and the Maslov index contribution vanishes  this gives
\begin{equation}
\label{DGformula}
-\sum_{\gamma_j}\frac{2 \delta}{2\pi |\det(I-P_{\gamma_j})|^{\frac{1}{2}}}(t-2\delta+\rmi 0)^{-1}+L^1_{\loc}.
\end{equation}

\end{proof}

In dimension two the dynamic of scattering billiard has been well studied (see, for instance, \cite{KozlovTreshchev,Sinai1970}). Following from the study of the Birkhoff Billiard in \cite[Chapter II]{KozlovTreshchev} one directly computes the Poincar\'e map and obtains the following corollary.

\begin{corollary}
	\label{2d sing}
If $\Omega$ is locally strictly convex near $\mathcal{B}_{\partial \Omega,\delta}$ and $d=2$ then the first singularity of $t\hat \xi_{\rel}(t)$ is of the form
\begin{equation*}
\sum_{\gamma_j}\frac{1}{2\pi} c_j (t-2\delta+ \rmi 0)^{-1}+L^1_{\loc},
\end{equation*}
where for each bouncing ball orbit $\gamma_j$ of length $2 \delta$ we have
$$
 c_j = \left(\frac{\delta r_j \rho_j }{r_j +\rho_j +\delta}\right)^{\frac{1}{2}}.
$$
Here $r_j$ and $\rho_j$ are the reciprocal of the curvatures of $\Omega$ at the two points of $\gamma_j \cap \partial \Omega$. 
\end{corollary}

In dimension three we could not find literature but the computation of the Poincar\'{e} map is straightforward and results in the following corollary.

\begin{corollary}
	\label{3d sing}
If $\Omega$ is locally strictly convex near $\mathcal{B}_{\partial \Omega,\delta}$ and $d=3$ then the first singularity of $t\hat \xi_{\rel}(t)$ is of the form
\begin{equation*}
\label{2Dfirstsingularity}
\sum_{\gamma_j}\frac{1}{2\pi} c_j (t-2\delta+ \rmi 0)^{-1}+L^1_{\loc},
\end{equation*}
where for each bouncing ball orbit $\gamma_j$ of length $2 \delta$ the coefficient $c_j$ is the geometric invariant of $\gamma_j$ given by
 \begin{equation*}
   c_j=\frac{\sqrt{\rho_{1} \rho_{2} r_{1} r_{2}}}{\sqrt{2 D}},
   \end{equation*}
where
\begin{gather}\nonumber
 D_j= 2 \delta ^2+2 \delta  \rho_{1}+2 \delta  \rho_{2}+2 \rho_{1} \rho_{2}+r_{1} \left(2 \delta +\rho_{1}+\rho_{2}+2 r_{2}\right)\\ +2 \delta  r_{2}-\left(\rho_{1}-\rho_{2}\right) \left(r_{1}-r_{2}\right) \cos (2 \theta )+\rho_{1} r_{2}+\rho_{2} r_{2} \label{Dref}
\end{gather}
and $r_{1},r_{2}$ are the radii of principal curvature at the first point in $\gamma_j \cap \partial \Omega$, 
$\rho_{1},\rho_{2}$ are the radii of principal curvature at the second point in $\gamma_j \cap \partial \Omega$.
Finally $\theta$ is the angle between the direction of the principal curvature corresponding to $r_{1}$ and the principal curvature corresponding to $\rho_{1}$.
\end{corollary}

Let $\Lsp(M)$ be the length spectrum of $M$ and it is given by
$$
\Lsp(M)=\{0\} \cup \{ t\in \R\, \big|\, \text{$|t|$ is the length of a closed broken geodesic in $M$} \}.
$$
 Note that this is the set stated in Theorem \ref{length spectrum thm}, which differs from the minimal length spectrum $\mathcal{L}(M)$ defined in the proof of Theorem \ref{length spectrum thm}. Recall that equation \eqref{wrel eqn} says
\begin{equation*}
\label{xirel1}
w_{\rel}(t)=w_{M}(t)-\sum_{j=1}^Nw_{M_j}(t).
\end{equation*}
On the other hand, the singularity of $w_{M}$ and $w_{M_j}$ can be analysed as in the Theorem \ref{length spectrum thm}. However, Theorem \ref{length spectrum thm} requires $\Omega_j$ to be strictly convex and it implies that
\begin{equation}
\label{xirel2}
\Sing(w_{M})\subset \Lsp(M) \qquad \text{and} \qquad \Sing(w_{M_j})\subset \Lsp(M_j).
\end{equation}

From the equation \eqref{wrel eqn} and relationship \eqref{xirel2}, one may think that there would be some cancellation of singularities and naively conjecture that $\Sing(w_\rel)\subset \Lsp(M)$.
This is not true in general. Therefore, we conclude the following remark.
\begin{rem}
In general, the singularities of $w_\rel$ are contained in $\mathcal{L}(M)\cup\mathcal{L}(M_1)\cup\cdots\cup\mathcal{L}(M_N)$. That is
$$
\Sing(w_\rel)\subset \mathcal{L}(M)\cup \mathcal{L}(M_1)\cup\cdots\cup\mathcal{L}(M_N).
$$
When $\Omega_j$'s are strictly convex, we have $\mathcal{L}(M_j)=\Lsp(M_j)=\{0\}$ for all $j$ and $\mathcal{L}(M)=\Lsp(M)$. In this case, Theorem \ref{smoothzerothm} tells us that
$$
\Sing(w_\rel)\subset \Lsp(M)\backslash\{0\}.
$$
\end{rem}

\

\textit{A natural question to ask is when
\begin{equation*}
\Lsp(M)\backslash\left(\Lsp(M_1)\cup\cdots\cup\Lsp(M_N)\right) \subset \Sing(w_\rel) .
\end{equation*}
}

\section{The function $\Xi$} \label{Section:4}

It was shown in  \cite{RTF} that 
$$
 R_\rel (\lambda) = (\Delta_X - \lambda^2)^{-1} - (\Delta_0 - \lambda^2)^{-1} - \sum_{j=1}^N \left(  (\Delta_{X_j} - \lambda^2)^{-1} - (\Delta_0 - \lambda^2)^{-1} \right)
$$
is a trace-class operator for any $\lambda$ in the upper half space. The function $\Xi$ is then uniquely determined by its decay along the positive imaginary axis and by
$$
 \Xi'(\lambda) = -2\lambda \tr\left( R_\rel(\lambda) \right). 
$$
The resolvent $(\Delta_X - \lambda^2)^{-1}$ can be represented as
$$
-2\rmi \lambda (\Delta_X - \lambda^2)^{-1} = \int_\R e^{\rmi \lambda | t |} \cos( t \Delta_X^\frac{1}{2}) \der t.
$$
Taking differences and the pointwise traces this implies
\begin{equation}
\label{Xiprime eqn}
 \rmi  \Xi'(\lambda) = -2\rmi \lambda \tr \left( R_\rel (\lambda) \right) =  \int_\R e^{\rmi \lambda | t |} w_\rel(t) \der t\\
\end{equation}
where the integral needs to be understood as a distributional pairing. The right hand side is well defined since $w_\rel$ is a tempered distribution supported away from $(-2 \delta,2\delta)$. From Theorem \ref{smoothzerothm}, we then obtain the following theorem.
\begin{theorem}
	\label{CapitalXi thm}
Let $\Omega$ be strictly locally strictly convex near $\mathcal{B}_{\partial \Omega,\delta}$ and let $\alpha>0$. Then for all $\lambda \in \C$ with
$\Im(\lambda) > \alpha| \Re(\lambda) |$ we have the bounds
\begin{equation*}
\left\{
\begin{aligned}
&\Xi'(\lambda)=-\sum_j\frac{2\rmi \delta}{|\det(I-P_{\gamma_j})|^{\frac{1}{2}}}e^{2\rmi \delta \lambda}+o( e^{-2 \delta \Im{\lambda}} )
\\
&\Xi(\lambda)=-\sum_j\frac{1}{|\det(I-P_{\gamma_j})|^{\frac{1}{2}}}e^{2\rmi \delta \lambda}+o(e^{- 2 \delta \Im{\lambda}})
\end{aligned}
\right.
\end{equation*}
where $\gamma_j$ are the shortest bouncing ball orbits.
\end{theorem}
\begin{proof}
From equation \eqref{Xiprime eqn}, we have
\begin{multline*}
\Xi'(\lambda)=-\frac{1}{2\pi}\int_0^{\infty} e^{\rmi \lambda t} \left[2\pi \rmi t \Im \left(\hat{\xi}_\rel(t)-\hat{\xi}_\rel(-t)\right)\right]\der t
\\
=\frac{1}{2\pi}\int_{-\infty}^{\infty} e^{\rmi \lambda t} \left(-4\pi \rmi\, t \theta (t) \Im \hat{\xi}_\rel (t)\right)\der t
\end{multline*}
and
\begin{equation*}
\Xi(\lambda)=\frac{1}{2\pi}\int_{-\infty}^{\infty} e^{\rmi \lambda t} \left(-4\pi  \theta \Im \hat{\xi}_\rel\right)(t)\der t,
\end{equation*}
where again the integrals are distributional pairings.
Using Theorem \ref{smoothzerothm}, our expression of $\Xi(\lambda)$ follows.
\end{proof}
Combining with Corollary \ref{2d sing} and Corollary \ref{3d sing} with Theorem \ref{CapitalXi thm}, we obtain the following corollaries.
\begin{corollary}
For two-strictly-convex-obstacle-scattering problems and $d=2$, one has 
\begin{equation*}
	\label{2ballscattering eqn}
\Xi(\lambda)=-\frac{1}{2}\left(\frac{r \rho}{\delta(r+\rho+\delta)}\right)^{\frac{1}{2}}e^{2\rmi \delta \lambda}+o(e^{- 2 \delta \Im{\lambda}}),
\end{equation*}
where $r$ and $\rho$ are the principal radii of curvature of $\Omega_{1}$ and $\Omega_{2}$ respectively at the reflection points.
\end{corollary}

\begin{corollary}
For two-strictly-convex-obstacle-scattering problems and $d=3$, we have 
\begin{equation*}
	\label{3ballscattering eqn}
\Xi(\lambda)=-\frac{1}{2}\sqrt{\frac{\rho _1 \rho _2 r_1 r_2}{2  D \delta^2}} e^{2\rmi \delta \lambda}+o(e^{- 2 \delta \Im{\lambda}}),
\end{equation*}
where $r_1,r_2$ and $\rho_1,\rho_2$ are the principal radii of curvature at the reflection points and $D$ is given by \eqref{Dref}.
\end{corollary}

In the case of two spheres of radius $r$ and $\rho$ we obtain as a special case
\begin{equation*}
	\label{3ballscattering eqn sphere}
\Xi(\lambda)=- \frac{r \rho}{4 \delta (r + \rho + \delta)} e^{2\rmi \delta \lambda}+o(e^{- 2 \delta \Im{\lambda}}).
\end{equation*}

\section*{Acknowledgements}
We are grateful to the anonymous referee for comments that improved the presentation of the paper.

\end{document}